\documentclass{article}[11pt]
\usepackage{amsthm,amsmath,a4wide}
\usepackage{amssymb}
\input xypic
\usepackage[all,tips]{xy}
\newtheorem{thm}{Theorem}[section]
\newtheorem{lem}[thm]{Lemma}
\newtheorem{prop}[thm]{Proposition}
\newtheorem{cor}[thm]{Corollary}
\newtheorem{conj}[thm]{Conjecture}

\theoremstyle{definition}
\newtheorem{defin}[thm]{Definition}

\theoremstyle{remark}
\newtheorem{remark}[thm]{Remark}
\newtheorem{example}[thm]{Example}

\newcommand{\bth}{\begin{thm}}
\renewcommand{\eth}{\end{thm}}
\newcommand{\bpr}{\begin{prop}}
\newcommand{\epr}{\end{prop}}
\newcommand{\ble}{\begin{lem}}
\newcommand{\ele}{\end{lem}}
\newcommand{\bco}{\begin{cor}}
\newcommand{\eco}{\end{cor}}
\newcommand{\bde}{\begin{defin}}
\newcommand{\ede}{\end{defin}}
\newcommand{\bex}{\begin{example}}
\newcommand{\eex}{\end{example}}
\newcommand{\bre}{\begin{remark}}
\newcommand{\ere}{\end{remark}}
\newcommand{\bcj}{\begin{conj}}
\newcommand{\ecj}{\end{conj}}
\newcommand{\beq}{\begin{equation}}
\newcommand{\eeq}{\end{equation}}
\newcommand{\ve}{{\varepsilon}}
\newcommand{\ot}{{\otimes}}
\newcommand{\be}{\begin{equation}}
\newcommand{\ee}{\end{equation}}

\newcommand{\lb}{\label}
\newcommand{\bpf}{\begin{proof}}
\newcommand{\epf}{\end{proof}}

\newcommand{\cross}{\cdot}
\newcommand{\g}{{\frak g}}

\newcommand{\ga}{{\frak a}}
\newcommand{\cG}{{\cal G}}

\newcommand{\Z}{{\bf Z}}

\newcommand{\Aut}{{\cal A}{\it ut}}

\newcommand{\Tr}{{\cal T}{\it r}}

\begin{document}
\title{Twisted derivations of Hopf algebras}
\author{A. Davydov}
\maketitle
\date{}
\begin{center}
Department of Mathematics and Statistics, University of New Hampshire, Durham, NH 03824,
USA
\end{center}
\begin{center}
alexei1davydov@gmail.com
\end{center}

\begin{abstract}
\noindent 
In the paper we introduce the notion of twisted derivation of a bialgebra.
Twisted derivations appear as infinitesimal symmetries of the category of representations.
More precisely they are infinitesimal versions of twisted automorphisms of bialgebras  \cite{da}.
Twisted derivations naturally form a Lie algebra (the tangent algebra of the group of twisted automorphisms).
Moreover this Lie algebra fits into a crossed module (tangent to the crossed module of twisted automorphisms).
Here we calculate this crossed module for universal enveloping algebras and for the Sweedler's Hopf algebra.
\end{abstract}
\tableofcontents

\section*{Introduction}
The paper studies infinitesimal symmetries of bialgebras which manifest themselves in representation theory.
It is very well-known that categories of representations (modules) of bialgebras are examples of so-called {\em
tensor categories}. It is less acknowledged that relations between bialgebras ({\em homomorphisms of bialgebras}) do
not capture all relations ({\em tensor functors}) between their representation categories. An algebraic notion which
does the job ({\em bi-Galois (co)algebra}) is known only to specialists. Fully representing tensor relations between
representation categories it is sometimes not very easy to work with. For example, composition of tensor functors
corresponds to tensor product of Galois (co)algebras and very often is quite tricky to calculate explicitly. At
the same time tensor functors of interest could have some additional properties which put restrictions on the
corresponding algebraic objects and allow one to have an alternative and perhaps simpler description.

Note that any representation category is equipped with a natural tensor functor to the category of vector spaces,
the functor forgetting the action of the bialgebra (the {\em forgetful functor}). In \cite{da} we dealt with tensor functors
between representation categories which preserve (not necessarily tensorly) the forgetful functors. The corresponding
algebraic notion is of {\em twisted homomorphism}. Composition of tensor functors correspond to composition operation on twisted homomorphism. Natural transformations of tensor functors correspond to certain relation which was called in \cite{da} 
({\em gauge transformations}).
The main object of study in \cite{da} was the category (groupoid) of twisted automorphisms of a Hopf algebra and its actions on (certain structures on) categories of representations.

In this paper we define the infinitesimal analog of the notion of twisted automorphism which we call {\em twisted derivation}. A twisted derivation of a bialgebra $H$ is a pair $(d,\phi)$ where $d:H\to H$ is an algebra derivation and $\phi$ is an element of $H\otimes H$, such that
$$(I\otimes d+d\otimes I)(\Delta(x)) - \Delta(d(x)) = [\phi,\Delta(x)],\quad \forall x\in H,$$
$$1\otimes\phi + (I\otimes\Delta)(\phi) = \phi\otimes 1 + (\Delta\otimes I)(\phi),$$
$$\ve d = 0,\quad (\varepsilon\otimes I)(\phi) = (I\otimes\varepsilon)(\phi) = 0.$$
Here $\Delta:H\to H\otimes H$ is the comultilication and $\varepsilon:H\to k$ is the counit of $H$. Twisted derivations form a Lie algebra $Der_{tw}(H)$ with the bracket
$$[(d,\phi),(d',\phi')] = ([d,d'], (d\otimes I+I\otimes d)(\phi') - (d'\otimes I+I\otimes d')(\phi) - [\phi,\phi']).$$ This is the tangent bracket to the group operation on twisted automorphisms.
An infinitesimal analog of gauge transformations of twisted automorphisms is the notion of {\em gauge transformation} of twisted derivations. A gauge transformation from a twisted derivation $(d,\phi)$ to a twisted derivation $(d',\phi')$ is an element $a$ of $H$ such that 
$$(d'-d)(x) = [a,x], \quad \forall x\in H,$$ $$\phi' - \phi = (a\otimes 1 + 1\otimes a) - \Delta(a).$$ Addition in $H$ gives rise to composition operation on gauge transformations turning twisted derivations into a {\em crossed module} of Lie algebras:
$$Der_{tw}(H)\stackrel{\partial}{\leftarrow} H.$$
Here $$\partial(a) = ([a,\ ],(a\otimes 1 + 1\otimes a) - \Delta(a)).$$
Recall (from \cite{g}) that a {\em crossed module of Lie algebras} is a homomorphism of Lie algebras $\partial:{\frak n}\to{\frak p}$ together with an action of $\frak p$ on $\frak n$ by derivations such 
that 
$$\partial(p(n)) = [p, \partial(n)] \quad\quad \forall p\in{\frak p}, n\in {\frak n}, $$
$$\partial(n)(m) = [n,m]  \quad\quad \forall  n,m\in {\frak n}.$$
The equivalence class of a crossed module ${\frak p}\stackrel{\partial}{\leftarrow}{\frak n}$ is controlled by two Lie algebras 
$$\pi_0({\frak p}\stackrel{\partial}{\leftarrow}{\frak n}) = coker(\partial),\quad\quad \pi_1({\frak p}\stackrel{\partial}{\leftarrow}{\frak n}) = ker(\partial)$$ and a (Lie algebra) cohomology class (the {\em Jacobiator})
$$J\in H^3(\pi_0,\pi_1).$$ See \cite{h1,h2,g} for details.

For the crossed module of twisted derivations $Der_{tw}(H)\stackrel{\partial}{\leftarrow} H$ we have associated two Lie algebras: the algebra of outer twisted derivations
$$OutDer_{tw}(H) = \pi_0(Der_{tw}(H)\stackrel{\partial}{\leftarrow} H)$$ 
and the abelian Lie algebra of {\em central primitive} elements
$$Z(H)\cap Prim(H) = \pi_1(Der_{tw}(H)\stackrel{\partial}{\leftarrow} H).$$ Here $Prim(H) = \{x\in H| \Delta(x) = x\otimes 1+1\otimes x\}$ and $Z(H) = \{x\in H|\ xy = yx\ \forall y\in H\}$. The Jacobiator is a cohomology class in $$H^3(OutDer_{tw}(H),Z(H)\cap Prim(H)).$$

As an example we treat the case of the universal enveloping algebra $U(\g)$ of a Lie algebra $\g$. It turns out that in characteristic zero any twisted derivation is gauge equivalent to a bialgebra derivation together with an invariant infinitesimal twist ({\em
separated} case). Gauge classes of invariant twists form an abelian Lie algebra isomorphic to $(\Lambda^2\g)^\g$, the invariant elements of
the exterior square of the Lie algebra $\g$. The Lie algebra of gauge classes of twisted derivations (the Lie algebra of outer twisted derivations) is a crossed product of the
Lie algebra of outer derivations of  $\g$ and the abelian Lie algebra  of invariant twists:
$$OutDer_{tw}(U(\g)) = OutDer(\g)\ltimes (\Lambda^2\g)^\g.$$
The space of central primitive elements of $U(\g)$ coincides with the centre of $\g$:
$$Z(U(\g))\cap Prim(U(\g)) = Z(\g).$$
We also compute the Jacobiator of $Der_{tw}(U(\g))$. In particular we show that its restriction to the abelian subalgebra of invariant twists $(\Lambda^2\g)^\g$ is zero.

We present a construction extending a bialgebra $H$ with an infinitesimal twist $\phi$ to a bialgebra $E(H)$ with a derivation $d:E(H)\to E(H)$ ($E(H)$ is the free algebra with derivation generated by the algebra $H$) such that $(d,\phi)$ is a twisted derivation of $E(H)$. This construction provides examples of non-separated twisted derivations.

We also look at twisted derivations of a non-commutative, non-cocommutative Sweedler's Hopf algebra $H_4$. Here again any twisted derivation is separated. The Lie algebra of bialgebra derivations and of invariant twists are both 1-dimensional. A non-zero bialgebra derivation acts non-trivially on invariant twists. Thus $OutDer_{tw}(H_4)$ is the 2-dimensional non-abelian Lie algebra. 

\section*{Acknowledgment}
The author would like to thank J. Cuadra and D. Nikshych for stimulating discussions.
Special thanks are to I. Owtscharenko for inspiration. 
 
\section{Twisted automorphisms and twisted derivations of bialgebras}
Throughout the paper $k$ be a ground field.

\subsection{Twisted automorphisms of bialgebras}
Here we recall (from \cite{da}) the notions of twisted automorphisms of a bialgebra,
their transformations and their relation with tensor autoequivalences of the category of representation.

A {\em twisted automorphism} of a bialgebra $H$ is a pair $(f,F)$,
where $f:H\to H$ is an algebra automorphism and $F$ is an invertible element of $H\otimes H$ (an {\em $f$-twist} or
simply {\em twist}) such that
\begin{equation}\label{conj}
F\Delta(f(x)) = (f\otimes f)(\Delta(x))F,\quad \forall x\in H,
\end{equation}
\begin{equation}\label{coc}
(F\otimes 1)(\Delta\otimes I)(F) = (1\otimes F)(I\otimes\Delta)(F),\quad \mbox{2-cocycle condition}
\end{equation}
\begin{equation}\label{norm}
\ve f = \ve,\quad (\varepsilon\otimes I)(F) = (I\otimes\varepsilon)(F) = 1,\quad \mbox{normalisation}.
\end{equation} 
Here and later on $\Delta:H\to H\otimes H$ is the coproduct and $\varepsilon:H\to k$ is the counit of $H$. 

For example, a bialgebra automorphism $f:H\to H$ is a twisted homomorphism with the identity twist $(f,1)$. A
twisted automorphism is {\em separated} if the first component $f:H\to H$ is a bialgebra automorphism. For a
separated twisted homomorphism the condition (\ref{conj}) amounts to the invariance of the twist with respect to the diagonal
sub-bialgebra $\Delta(H)\subset H\otimes H$: $$\Delta(f(x))F = F\Delta(f(x)),\quad \forall x\in H.$$ 

The {\em composition} of twisted automorphisms $(f,F)$ and $(f',F')$ is
\begin{equation}\label{comp}(f,F)\circ (f',F') = (ff',f(F')F).\end{equation} Here $f(F') = (f\otimes f)(F')$. It is not hard to verify that the result is a
twisted automorphism and that the composition is associative.
Note that separated twisted automorphisms are closed under composition.

By a {\em gauge transformation} $(f,F)\to(f',F')$ of twisted automorphisms $(f,F),(f',F')$ of $H$ we will mean an invertible element $a$ of $H$ such that
\begin{equation}\label{comgt}
af(x) = f'(x)a,\quad \forall x\in H,
\end{equation}
\begin{equation}\label{mongt}
F'\Delta(a) = (a\otimes a)F.
\end{equation}
We will depict it graphically as follows: 

$$
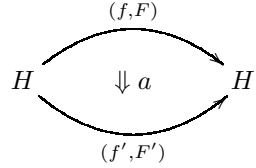
\xymatrix{
H {\ar@/^20pt/[rr]^{(f,F)} }
{\ar@/_20pt/[rr]_{(f',F')}  } & \Downarrow a & H \\
} $$
 Note that the condition
(\ref{mongt}) together with normalisation conditions for $F$ and $G$ implies $\varepsilon(a) = 1$. Note also that a
twisted homomorphism gauge isomorphic to a separated twisted homomorphism is not necessarily separated.

For successive gauge transformations
$$
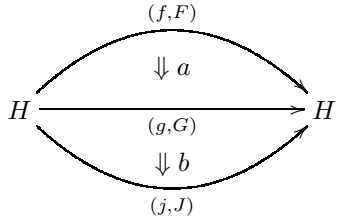
\xygraph{ !{0;/r3.2pc/:;/u4.3pc/::}[]*+{H} 
(
:@/^30pt/[r(3)]*+{H}="r" ^{(f,F)}
,
:@/_30pt/"r" _{(j,J)}
,
:@{}[r(1.5)u(.3)]*+{\Downarrow a}
,
:@{}[r(1.5)d(.4)]*+{\Downarrow b}
,
:"r" _{(g,G)}
) 
}$$
the composition $b.a:(f,F)\to(j,J)$ is simply the product $ba$ in $H'$. Again it is quite
straightforward to check that this is a transformation.

We can also define compositions of transformations and twisted automorphisms in the following two situations:
$$
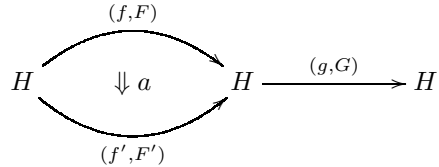
\xymatrix{
H {\ar@/^20pt/[rr]^{(f,F)}} {\ar@/_20pt/[rr]_{(f',F')}} &  \Downarrow a &
H \ar[rr]^{(g,G)} & & H  }$$
$$\xymatrix{H \ar[rr]^{(f,F)} & & H {\ar@/^20pt/[rr]^{(g,G)}}
{\ar@/_20pt/[rr]_{(g',G')}} &  \Downarrow b & H }$$ we define it to be
$(g,G)\circ a = g(a)$ in the first case and $b\circ(f,F) = b$ in the second. The following properties intertwining
compositions of twisted homomorphisms and gauge transformations are quite straightforward consequences of the
definitions: $$(a.b)\circ(f,F) = (a\circ(f,F)).(b\circ(f,F)),$$ $$(g,G)\circ(a.b) = ((g,G)\circ a).((g,G)\circ b),$$
$$(a\circ(f,F)).((g,G)\circ b) = ((g,G)\circ b).(a\circ(f,F)).$$ 
Note (see also \cite{da} for details) that the structures described above turn twisted automorphisms and their gauge transformations into a categorical group $\Aut_{tw}(H)$.

Recall (see \cite{bs} for the history) that Cat-groups are the same as crossed modules of Whitehead
(\cite{wh}). A {\em crossed module} of groups is a pair of groups $P,C$ with a (left) action of $P$ on $C$
(by group automorphisms): $$P\times C\to C,\quad (p,c)\mapsto ^{p}{c}$$ and a homomorphism of groups $$P
\stackrel{\partial}{\leftarrow} C$$ such that $$\partial(^{p}{c}) = p\partial(c)p^{-1},\quad  ^{\partial(c)}{c'} =
cc'c^{-1}.$$ 
A {\em
map} of crossed modules $(P,C)\to (E,N)$ is a triple $(\tau,\nu,\theta)$ where $\tau$ and $\nu$ are maps making the
diagram $$\xymatrix{P \ar[d]^\tau & C\ar[l]^{\partial} \ar[d]^\nu \\ E & N \ar[l]^{\partial} }$$ commutative, and
$\theta:P\times P\to N$ is such that $$\tau(pq) = \partial(\theta(p,q))\tau(p)\tau(q),\quad p,q\in P,$$ $$\nu(ab) =
\theta(\partial(a),\partial(b))\nu(a)\nu(b),\quad a,b\in C,$$ $$\theta(p,qr) {^{\tau(p)}{\theta(q,r)}} =
\theta(pq,r)\theta(p,q),\quad p,q,r\in P,$$ $$\theta(1,q) = 1 = \theta(p,1),\quad p,q\in P,$$ $$\nu( ^{p}{a}) =
\theta(p,\partial(a)) ^{\tau(p)}{\nu(a)}.$$

Complete invariants of a categorical-group $\cG$ with respect to monoidal equivalences are $$\pi_0(\cG),\ \ \pi_1(\cG),\ \ 
\phi\in H^3(\pi_0(\cG),\pi_1(\cG)),$$ where the first is the group of isomorphism classes of objects, the second is the
abelian group ($\pi_0(\cG)$-module) $Aut_\cG(I)$ of automorphisms of the unit object and the third is a cohomology
class (the {\em associator}). In the crossed module setting $$\pi_0 = coker(\partial),\ \ \pi_1 = ker(\partial).$$ Note
that the image of $\partial$ is normal so the cokernal has sense. The class $\phi$ is defined as follows: choose a
section $\sigma:coker(\partial)\to P$ and a map $a:coker(\partial)\times coker(\partial)\to C$ such that $$\sigma(fg) =
\partial(a(f,g))\sigma(f)\sigma(g),\quad  f,g\in coker(\partial).$$ Then for any $f,g,h\in coker(\partial)$ the
expression $$a(f,gh) ^{\sigma(f)}{a(g,h)}a(f,g)^{-1}a(fg,h)$$ is always in the kernel  of $\partial$ and is a group
3-cocycle of $coker(\partial)$ with coefficients in $ker(\partial)$. The cohomology class $\phi$ of this 3-cocycle does not depend on the choices made.

The crossed module of groups corresponding to the categorical group of twisted automorphisms has
the form 
\begin{equation}\label{gcm}
Aut_{tw}(H) \stackrel{\partial}{\leftarrow} H_\varepsilon^\cross.
\end{equation} 
Here $Aut_{tw}(H)$ is the group of
twisted automorphisms of $H$ with respect to the composition, $H_\varepsilon^\cross$ is the group of invertible
elements $x$ of $H$ such that $\varepsilon(x)=1$, and $\partial$ sends $x$ into the pair (an {\em inner} twisted
automorphism) $( ^{x}{(\ )},(x\otimes x)\Delta(x)^{-1})$ where the first component is the conjugation automorphism:
$$^{x}{(\ )}:H\to H,\quad ^{x}{(y)} = xyx^{-1}.$$ The action of $Aut_{tw}(H)$ on $H^\cross$ is given by the action of
the first component $(f,F)(y) = f(y)$.

There are two important Cat-subgroups in $\Aut_{tw}(H)$. The first is the full Cat-subgroup $\Aut^1_{tw}(H)$ of
twisted automorphisms with the identity as the first component. Its crossed module is $$Aut^1_{tw}(H)
\stackrel{\partial}{\leftarrow} (Z(H)_\varepsilon)^\cross.$$ Here $Aut^1_{tw}(H)$ is the group of {\em invariant
twists} on $H$ (invertible elements of $H\otimes H$ commuting with the image $\Delta(H)$ and satisfying the 2-cocycle
condition), $(Z(H)_\varepsilon)^\cross$ is the group of invertible elements of the centre of counit 1:
$\varepsilon(x)=1$ . Again $\partial$ assigns to $x$ the invariant twist $(x\otimes x)\Delta(x)^{-1}$. The action of
$Aut^1_{tw}(H)$ on $(Z(H)_\varepsilon)^\cross$ is trivial.

The second is the full Cat-subgroup $\Aut_{bialg}(H)$ of bialgebra automorphisms of $H$. Here the crossed module is
$$Aut_{bialg}(H) \stackrel{\partial}{\leftarrow} G(H),$$ where $Aut_{bialg}(H)$ is the group of automorphisms of $H$ as
a bialgebra, $G(H)=\{x\in H,\quad \Delta(x) = x\otimes x\}$ is the group of group-like elements of $H$ and $\partial$
sends $x$ into the conjugation automorphism. The action of $Aut_{bialg}(H)$ on $G(H)$ is obvious.

Note that the Cat-subgroup $\Aut^1_{tw}(H)$ is what might be called {\em normal}: the components of its crossed module
are normal subgroups in the components of the crossed module for $Aut_{tw}(H)$ and the action of $Aut_{tw}(H)$ on
$H^\cross$ preserve the subgroup $Z(H)^\cross$.

The Cat-subgroup $\Aut_{bialg}(H)$ is not in general normal. In \cite{da} it was characterised as the stabiliser
of a certain action. 

\bre
Recall from \cite{da} the following categorical interpretation of twisted automorphisms and their gauge transformations.
For a homomorphism of algebras $f:H\to H'$ there is defined the {\em inverse image} functor $f^*:H'-Mod\to H-Mod$, which
turns an $H'$-module $M$ into an $H$-module $f^*(M)$. As a vector space, $f^*(M)$ is the same as $M$ but with a
new module structure $x.m = f(x)m$ for $x\in H$ and $m\in M$. 

For a twisted automorphism $(f,F):H\to H$ the inverse image functor $f^*:H-Mod\to H-Mod$ becomes tensor, with the
tensor structure given by multiplication by the twist: $$F_{M,N}:f^*(M\otimes H)\to f^*(M)\otimes f^*(N),\quad
m\otimes n\mapsto F(m\otimes n).$$ Compositions of twisted homomorphisms and corresponding functors are related as
follows: $$((f,F)\circ(g,G))^* = (g,G)^*\circ(f,F)^*.$$ A gauge transformation $a:(f,F)\to (g,G)$ defines a tensor
natural transformation $a:(f,F)^*\to (g,G)^*$: $$a_M:f^*(M)\to g^*(M),\quad m\mapsto am.$$ Compositions of gauge
transformations correspond to compositions of natural transformations.

It is straightforward to see that the condition (\ref{conj}) guarantees $H$-linearity of the tensor structure
$F_{M,N}$ while the 2-cocycle condition for $F$ is equivalent to the coherence axiom for the tensor structure. Similarly, the
condition (\ref{comgt}) for a gauge transformation $a$ says that $a_M$ is a morphism of $H$-modules and the condition
(\ref{mongt}) is equivalent to the monoidality of $a_M$.
\ere

\subsection{Infinitesimal twisted automorphisms}

Here we look at infinitesimal twisted automorphisms and their guage transformations. This leads us to the notion of twisted derivation. 

Let $h$ be the dual number, i.e. $h^2=0$. Let $(f,F)$ be a twisted automorphism of a bialgebra $H$, defined over the algebra of dual numbers $k[h]$, such that
\begin{equation}\label{inft}
f = I + h d,\quad F = 1\otimes 1 + h \phi,
\end{equation} 
for a (multiplicative) derivation $d:H\to H$ and an element $\phi\in H^{\otimes 2}$.
\begin{lem}
The defining conditions (\ref{conj},\ref{coc},\ref{norm}) of a twisted automorphism are equivalent to the following equations on a derivation $d$ and an element $\phi$:
\begin{equation}\label{conjd}(I\otimes d+d\otimes I)(\Delta(x)) - \Delta(d(x)) = [\phi,\Delta(x)],\quad \forall x\in H,\end{equation}
\begin{equation}\label{cocd}1\otimes\phi + (I\otimes\Delta)(\phi) = \phi\otimes 1 + (\Delta\otimes I)(\phi),\end{equation}
\begin{equation}\label{normd}\ve d = 0,\quad (\varepsilon\otimes I)(\phi) = (I\otimes\varepsilon)(\phi) = 0.\end{equation}
\end{lem}
\begin{proof}
The equations (\ref{conjd},\ref{cocd},\ref{normd}) follow from the comparison of coefficients for $h$ in the equations (\ref{conj},\ref{coc},\ref{norm}) respectively.
\end{proof}

We call a pair $(d,\phi)$ satisfying (\ref{conjd},\ref{cocd},\ref{normd}) a {\em twisted derivation} of the bialgebra $H$. Denote by $Der_{tw}(H)$ the set of twisted derivations of bialgebra $H$.

Call a twisted automorphism (\ref{inft}) the {\em infinitesimal twisted automorphism} corresponding to a twisted derivation $(d,\phi)$.
Composition of infinitesimal twisted automorphisms corresponds to addition of twisted derivations:
$$(I + h d,1\otimes 1 + h \phi)\circ(I + h d',1\otimes 1 + h \phi') = (I + h (d+d'),1\otimes 1 + h(\phi+\phi')).$$

The infinitesimal twisted automorphism corresponding to a twisted derivation $(d,\phi)$ is {\em separated} if $d$ is a bialgebra derivation:
$$(I\otimes d+d\otimes I)(\Delta(x)) = \Delta(d(x)).$$ Note that $\phi$ for a separated twisted derivation is invariant with respect to $\Delta(H)$ and that the pair $(0,\phi)$ is a twisted derivation.  
Note also that the notion of separated twisted derivation is not gauge invariant. We call a twisted derivation {\em separable} if it is gauge equivalent to a separated one. 

Now suppose that $h^3=0$ and consider the next order infinitesimal twisted automorphisms of $H$, i.e. twisted automorphisms of the form
$$f = I + h d + h^2\delta,\quad F = 1\otimes 1 + h \phi + h^2\psi.$$
The commutator of two such twisted automorphisms has a form
$$[(I + h d + h^2\delta,1\otimes 1 + h \phi + h^2\psi),(I + h d' + h^2\delta',1\otimes 1 + h \phi' + h^2\psi')] =$$ $$\big(h^2[d,d'], h^2\big((d\otimes I+I\otimes d)(\phi') - (d'\otimes I+I\otimes d')(\phi) - [\phi,\phi']\big)\big).$$
This in particular shows that the operation 
\be\lb{libr}[(d,\phi),(d',\phi')] = ([d,d'], (d\otimes I+I\otimes d)(\phi') - (d'\otimes I+I\otimes d')(\phi) - [\phi,\phi'])\ee is a Lie bracket on the vector space $Der_{tw}(H)$ of twisted derivations. Of course this fact can be checked directly.

Now we go back to first order infinitesimal twisted automorphisms (assuming again that $h^2=0$). Consider two such twisted automorphisms $(f,F), (f',F')$, corresponding to twisted derivations $(d,\phi)$ and $(d',\phi')$. Let $a:(f,F)\to(f',F')$ be a gauge transformation of the form
$$a = 1 + h \alpha,\quad \alpha\in H.$$
\begin{lem}
The conditions (\ref{comgt},\ref{mongt}) are equivalent to the following
$$(d'-d)(x) = [a,x], \quad \forall x\in H,$$ \be\label{gt}\phi' - \phi = (a\otimes 1 + 1\otimes a) - \Delta(a).\ee 
\end{lem}
\begin{proof}
Follow from the comparison of coefficients for $h$ in the equations (\ref{comgt},\ref{mongt}) respectively.
\end{proof}
We call $a\in H$ satisfying the conditions (\ref{gt}) a {\em gauge transformation} between twisted derivations $(d,\phi)$ and $(d',\phi')$.

Define a map $\partial:H_\varepsilon\to Der_{tw}(H)$ by
$$\partial(a) = ([a,\ ],(a\otimes 1 + 1\otimes a) - \Delta(a)).$$ Here $H_\varepsilon = ker(\varepsilon)$ is the kernel of the counit (the augmentation ideal). 
The above calculations show that this map is a crossed module of Lie algebras with respect to the $Der_{tw}(H)$-action on $H$:
$$(d,\phi)(\alpha) = d(\alpha),\quad \alpha\in H.$$
More precisely the following statement is established by the calculations of this section.
\bth
The crossed module of Lie algebras $\partial:H_\varepsilon\to Der_{tw}(H)$ is tangent to the crossed module of groups (\ref{gcm}):
$$Aut_{tw}(H) \stackrel{\partial}{\leftarrow} H_\varepsilon^\cross.$$
\eth

Note the $ker(\partial)$ coincides with the space $Z(H)\cap Prim(H)$ of central primitive elements of $H$.
We call the Lie algebra
$$OutDer_{tw}(H) = coker(\partial)$$ the algebra of {\em outer twisted derivations}.

There are two important sub-crossed modules of Lie algebras in $Der_{tw}(H)$. The first is the sub-crossed module $$Der^0_{tw}(H)
\stackrel{\partial}{\leftarrow} Z(H)_\varepsilon.$$ Here $Der^0_{tw}(H)$ is the Lie algebra (with respect to commutator) of {\em invariant
infinitesimal twists} on $H$ (elements of $H_\ve\otimes H_\ve$ commuting with the image $\Delta(H)$ and satisfying the 2-cocycle
condition (ref{})), and $Z(H)_\varepsilon = Z(H)\cap H_\varepsilon$. The map $\partial$ assigns to $x$ the invariant infinitesimal twist $(x\otimes x)-\Delta(x)$. The action of
$Der^0_{tw}(H)$ on $Z(H)_\varepsilon$ is trivial. 
We denote by $$OutDer^0_{tw}(H) = coker(\partial:Z(H)_\varepsilon\to Der^0_{tw}(H))$$ the Lie algebra of {\em outer invariant infinitesimal twists}.

The second is the sub-crossed module of bialgebra derivations of $H$:
$$Der_{bialg}(H) \stackrel{\partial}{\leftarrow} Prim(H),$$ where $Der_{bialg}(H)$ is the Lie algebra of derivations of $H$ as
a bialgebra. The map $\partial$
sends $x$ into the conjugation derivation $[x,-]$. The action of $Der_{bialg}(H)$ on $Prim(H)$ is the obvious one.
We denote by $$OutDer_{bialg}(H) = coker(\partial:Prim(H)\to Der_{bialg}(H))$$ the Lie algebra of {\em outer bialgebra derivations}.

Note that the crossed submodule $Der^0_{tw}(H)$ is what might be called {\em normal}: the components 
are Lie ideals in the components of the crossed module $Der_{tw}(H)$ and the action of $Der_{tw}(H)$ on
$H$ preserve the subspace $Z(H)$.

In the case when any twisted derivation is gauge equivalent to a separated one {the {\em separated case}) the Lie algebra of outer twisted derivations is the semi-direct product:
$$OutDer_{tw}(H) = OutDer_{bialg}(H)\ltimes OutDer^0_{tw}(H).$$

\subsection{Twisted derivations of quasi-triangular bialgebras}

Recall that a {\em quasi-triangular structure} on a bialgebra $H$ is an invertible element $R\in H\otimes H$ (a {\em
universal R-matrix}) satisfying
\begin{equation}\label{conjr}
Rt\Delta (x) = \Delta (x)R \quad \forall x\in H,
\end{equation}
along with {\em triangle equations}: 
\beq\lb{treq}
(I\otimes \Delta)(R) = R_{23}R_{13},\quad\quad (\Delta\otimes I)(R) = R_{12}R_{13}.
\eeq
Here $R_{12} =
R\otimes 1$, $R_{13} = (I\otimes t)(R_{12})$ etc., where $t:H\ot H\to H\ot H$ denotes the transposition of tensor factors.

Denote by $\Tr(H)$ the set of universal R-matrices.

\bre
Recall that a triangular structure $R$ on a bialgebra $H$ defines a braiding: $$c_{M,N}:M\otimes N\to N\otimes N,\quad
m\otimes n\mapsto R(n\otimes m)$$ on the category $H-Mod$. 

Indeed the condition (\ref{conjr}) implies that $c_{M,N}$ is a morphism of $H$-modules: $$c_{M,N}(\Delta(x)(m\otimes
n)) = Rt\Delta(x)(n\otimes m) = \Delta(x)R(n\otimes m) = \Delta(x)c_{M,N}(m\otimes n).$$ The triangle equations are
equivalent to the hexagon axioms for the braiding.
\ere

The interpretation of quasi-triangular structures as braidings allows us to define an action of twisted automorphisms on quasi-triangular structures. 
For a twisted automorphism $(f,F)$ and a quasi-triangular structure $R$ on $H$ the {\em twisted} quasi-triangular structure is defined as follows:
\beq\lb{twr}
{R}^{(f,F)}= F^{-1}(f\ot f)(R)F_{21}.
\eeq
It is straightforward to verify that the properties of the $R$-matrix are
preserved. Moreover, gauge isomorphic twisted automorphisms act equally. Indeed, for $g(x) = af(x)a^{-1}$ and $G = (a\otimes a)F\Delta(a)^{-1}$, 
$${R}^{(g,G)} = \Delta(a)F^{-1}(a\otimes a)^{-1}(a\otimes a)(f\ot f)(R)(a\otimes a)^{-1}(a\otimes a)F_{21}t\Delta(a)^{-1} =$$ $$ = \Delta(a)F^{-1}(f\ot f)(R)F_{21}t\Delta(a)^{-1} = {R}^{(f,F)}.$$ 
Thus an action of
the group $Out_{tw}(H)$ on the set $\Tr(H)$ of universal R-matrices is defined. 
For $R\in\Tr(H)$ denote by $Out_{tw}(H,R)$ the stabiliser $St_{Out_{tw}(H)}(R)$:
\beq\lb{str}Out_{tw}(H,R) = \{(f,F)\in Out_{tw}(H)|\ (f\otimes f)(R)F_{21} = FR\}.\eeq
\bre
Note that the group $Out_{tw}(H,R)$ naturally embeds in to the group of isomorphism classes of braided tensor autoequivalences of $H-Mod$ with the braiding given by $R$.
\ere

Here we look at the infinitesimal version of the action of twisted automorphisms on R-matrices. 
The tangent space $T_R(\Tr(H))$ to the space $\Tr(H)$ of universal R-matrices at a point $R\in \Tr(H)$ is the vector space 
\beq\lb{ttr}
\{r\in H^{\ot 2}|\ (I\otimes \Delta)(r) = r_{23}R_{13}+R_{23}r_{13},\quad(\Delta\otimes I)(r) = r_{12}R_{13}+R_{12}r_{13},\quad \Delta (x)r = rt\Delta (x) \quad \forall x\in H\}.
\eeq
Indeed, writing an R-matrix as $R+hr$ with $h^2=0$ transforms \eqref{conjr} and \eqref{treq} into the defining conditions in \eqref{ttr}. 
Similarly the tangent Lie algebra of the stabiliser \eqref{str} is 
$$OutDer_{tw}(H,R) = \{(d,\phi)\in OutDer_{tw}(H)|\ (I\otimes d+d\otimes I)(R) = \phi_{21}R-R\phi\}.$$
Expanding the ingredients of \eqref{twr} in powers of $h$ with $h^3=0$ and looking at the change in the coefficient for $h^2$ we get the following formula for the action of twisted derivations on infinitesimal R-matrices:
\beq\lb{twrc}
(d,\phi)(r) = (I\otimes d+d\otimes I)(r) - \phi r + r\phi_{21}.
\eeq
The above implies to the following
\bpr
The formula \eqref{twrc} defines the structure of an $OutDer_{tw}(H,R)$-module on $T_R(\Tr(H))$.
\epr

\bre
Let $H$ be a cocommutative bialgebra. 
It is known that infinitesimal quasi-triangular structures (i.e. $R$-matrices of the form $R=1+hr$) on $H$ correspond to invariant {\em classical R-matrices}, i.e. elements $r\in H^{\ot 2}$ satisfying
$$
r\Delta (x) = \Delta (x)r \quad \forall x\in H
$$
and
$$
(I\otimes \Delta)(r) = r_{13}+r_{12},\quad\quad (\Delta\otimes I)(r) = r_{13}+r_{23},$$
Note that the last two equations imply that $r$ belongs to $Prim(H)^{\ot 2}$. 
Thus the tangent space to the space of universal R-matrices at a point $1\in \Tr(H)$ is $\Tr_{inf}(H) = (Prim(H)^{\ot 2})^H$ the vector space of invariant classical R-matrices of $H$. 
The tangent Lie algebra of the stabiliser of the canonical quasi-triangular structure on $H$ (the one corresponding to $1\in H^{\ot 2}$) is
$$OutDer_{tw}(H,1) = \{(d,\phi)\in OutDer_{tw}(H)|\ \phi=\phi_{21}\}.$$ 
The $OutDer_{tw}(H,1)$-action on $(Prim(H)^{\ot 2})^H$ has the form
$$(d,\phi)(r) = (I\otimes d+d\otimes I)(r) - [\phi,r].$$
\ere

\section{Examples}

\subsection{Twisted derivations of universal enveloping algebras}\label{invtwi}
In this section we assume that the ground field $k$ is of
characteristic zero.  

Now let $(d,\phi)$ be a twisted derivation of $U(\g)$. In particular
$\phi\in U(\g)^{\otimes 2}$ is an infinitesimal (or co-Hochschild) 2-cocycle (see Appendix \ref{coho}), i.e.:
$$1\otimes \phi + (I\otimes\Delta)(\phi) = \phi\otimes 1 + (\Delta\otimes I)(\phi).$$ Proposition \ref{chu} implies that there is $a\in U(\g)$ such that
\begin{equation}\label{cob}
\phi = \overline{\phi} + a\otimes 1 + 1\otimes a - \Delta(a),\quad \overline{\phi} = Alt_2(\phi) = \frac{1}{2}(\phi - \phi_{21}).
\end{equation}
Thus the twisted derivation $(d,\phi)$ is gauge equivalent to a twisted derivation $(d',\overline\phi)$ (with $\overline\phi\in\Lambda^2\g$):
$$(d,\phi) = (d + [a,-] - [a,-],\overline\phi + a\otimes 1 + 1\otimes a - \Delta(a)) = (d',\overline\phi) - \partial(a).$$

Now look at the equation (\ref{conjd}) for the twisted derivation $(d',\overline\phi)$:
 $$(I\otimes d'+d'\otimes I)(\Delta(x)) - \Delta(d'(x)) = [\overline\phi,\Delta(x)],$$
The left hand side is a symmetric element of $U(\g)^{\ot 2}$ (due to cocommutativity of $U(\g)$), while the right hand side is anti-symmetric:
$$[\overline\phi,\Delta(x)]\in \Lambda^2\g$$ Thus they both must be zero:
$$(I\otimes d'+d'\otimes I)(\Delta(x)) - \Delta(d'(x)) = 0,$$
$$[\overline\phi,\Delta(x)] = 0,$$
i.e. $(d',\overline\phi)$ is a separated twisted derivation.

Thus we have the following.
\bth\lb{tdu}
Let $\g$ be a Lie algebra over a field of characteristic zero. 
Then the Lie algebra of outer twisted derivations of $U(\g)$ has a form:
$$OutDer_{tw}(U(\g))\simeq OutDer(\g)\ltimes (\Lambda^2\g)^\g,$$ where the crossed product is taken with respect to the natural action of $OutDer(\g)$ on the abelian Lie algebra $(\Lambda^2\g)^\g$.
\eth
\bpf
We have established that any twisted derivation of $U(\g)$ is gauge equivalent to $(d,\phi)$, where $d$ is a bialgebra derivations and $\phi\in(\Lambda^2\g)^\g$. A bialgebra derivation of $U(\g)$ is (induced by) a Lie algebra derivation of $\g$. Now all we need to show is that the gauge class of $(d,\phi)$ depends only on the class of $d$ in the Lie algebra of outer derivations $OutDer(\g)$. This can be seen by applying gauge equivalences corresponding to $x\in\g\subset U(\g)$:
$$(d,\phi) + \partial(x) = (d + [x,-],\phi).$$

The commutator $[X,Y]$ (in $U(\g)^{\ot 2}$) of two skew-symmetric elements is necessarily symmetric. Thus the commutator on $(\Lambda^2\g)^\g$ is trivial. It follows from (\ref{libr}) that the action of $OutDer(\g)$ on $(\Lambda^2\g)^\g$ is $$d(X) = (d\ot 1+1\ot d)(X).$$
\epf

In particular the theorem implies that the Lie algebra of invariant infinitesimal twists coincides with $(\Lambda^2\g)^\g.$ This of course can be seen directly. Indeed, if $\phi\in U(\g)^{\ot 2}$  is a $\g$-invariant infinitesimal twist then $\overline{\phi}$ is also $\g$-invariant which altogether makes $$\partial(a) = a\otimes 1 + 1\otimes a - \Delta(a) = \phi - \overline\phi$$ $\g$-invariant. According to lemma \ref{invco} the last implies that $a$ can be chosen to be $\g$-invariant (central).

\bre
Here we say a few words about the Jacobiator of the crossed module of Lie algebras $$Der_{tw}(U(\g)) \stackrel{\partial}{\leftarrow} U(\g)$$ 
as a cohomology class in $$H^3\big(OutDer_{tw}(U(\g)),Z(\g)\big) = H^3\big(OutDer(\g)\ltimes (\Lambda^2\g)^\g,Z(\g)\big).$$ 
We start by showing that its restriction to $H^3((\Lambda^2\g)^\g,Z(\g))$, which coincides with the Jacobiator of the crossed module of Lie algebras \be\lb{cml1}Z^2(U(\g))^\g \stackrel{\partial}{\leftarrow} C^1(U(\g))^\g\ee is trivial. Here $Z^2$ is the Lie algebra (with respect to the commutator in
$U(\g)^{\otimes 2}$) of 2-cocycles of the subcomplex of $\g$-invariants of (\ref{tangcoh}) and $C^1(U(\g))^\g = Z(U(\g))$ is the abelian Lie
algebra of 1-cochains of the same subcomplex. Note that the action of $Z^2(U(\g))^\g$ on $C^1(U(\g))^\g$ is trivial and the commutator
$[X,\partial(a)]$ is zero for any $X\in Z^2(U(\g))^\g$ and $a\in Z(U(\g))$ (thus in particular fulfilling the axioms of crossed module of Lie
algebras). 
Thus for any $X,Y\in Z^2(U(\g))^\g$ we have $[X,Y] = [Alt_2(X),Alt_2(Y)]$. As was noted in the proof of theorem \ref{tdu} the commutator $[Alt_2(X),Alt_2(Y)]$ is a symmetric 2-cocyle and hence a coboundary. So we can write it as $\partial(a(Alt_2(X),Alt_2(Y)))$ for some function $a:(\Lambda^2\g)^\g\ot(\Lambda^2\g)^\g\to Z(U(\g))$. The function $a$ can be extended  to a function $a:Z^2(U(\g))^\g\ot Z^2(U(\g))^\g\to Z(U(\g))$ by $a(X,Y) = a(Alt_2(X),Alt_2(Y))$. Now we
have $a(X,[Y,Z]) = 0$, so the Jacobiator of the crossed module of Lie algebras (\ref{cml1}) is trivial.
\newline
It follows from the spectral sequence of the extension
$$(\Lambda^2\g)^\g\to OutDer_{tw}(U(\g))\to OutDer(\g)$$ that the cohomology $E^3_\infty = H^3\big(OutDer_{tw}(U(\g)),Z(\g)\big)$ has a filtration 
$$E^3_\infty = F^0E^3_\infty \supset F^1E^3_\infty \supset F^2E^3_\infty \supset F^3E^3_\infty$$
with the first associated quotient $$F^0E^3_\infty/F^1E^3_\infty = E^{0,3}_\infty \subset E^{0,3}_2 = H^3((\Lambda^2\g)^\g,Z(\g)).$$ Since the induced map 
$$H^3\big(OutDer_{tw}(U(\g)),Z(\g)\big) = E^3_\infty\to E^{0,3}_2 = H^3((\Lambda^2\g)^\g,Z(\g))$$ is the restriction we have shown that the Jacobiator belongs to $F^1E^3_\infty$. The second quotient is
$$F^1E^3_\infty /F^2E^3_\infty = E^{1,2}_\infty \subset E^{1,2}_2 = H^1\big(OutDer(\g),H^2((\Lambda^2\g)^\g,Z(\g))\big).$$ The image of the Jacobiator under the homomorphism $$F^1E^3_\infty \to H^1\big(OutDer(\g),H^2((\Lambda^2\g)^\g,Z(\g))\big)$$
can be described as follows. Let $d\in Der(\g), X,Y\in (\Lambda^2\g)^\g$. Since
$$(d\ot I+I\ot d)([X,Y]) = (d\ot I+I\ot d)(\partial(a(X,Y))) = \partial(d(a(X,Y)))$$ coincides with
$$[(d\ot I+I\ot d)(X),Y] + [X,(d\ot I+I\ot d)(Y)] = \partial\big(a((d\ot I+I\ot d)(X),Y)+a(X,(d\ot I+I\ot d)(Y))\big)$$
the difference 
$$a(d,X,Y) = a\big((d\ot I+I\ot d)(X),Y\big)+a\big(X,(d\ot I+I\ot d)(Y)\big) - d\big(a(X,Y)\big)$$ belongs to $Z(\g)$. The assignment 
$$d\mapsto \big(X,Y\mapsto a(d,X,Y)\big)$$ defines a class in $H^1\big(OutDer(\g),H^2((\Lambda^2\g)^\g,Z(\g))\big)$, which is the class of Jacobiator. 
\newline
There is a choice of $a(X,Y)$ which makes $a(d,X,Y)$ identically zero. Indeed $a$ is the second component of the co-chain homotopy from remark \ref{chho}. The $\g$-invariance of the homotopy implies that $a(d,X,Y)$ is zero.

In general if the class of $a(d,X,Y)$ in $H^1\big(OutDer(\g),H^2((\Lambda^2\g)^\g,Z(\g))\big)$ is trivial, the Jacobiator belongs to $F^2E^3_\infty$. The obstruction for the Jacobiator to be in $F^3E^3_\infty$ is measured by its image in 
$$F^2E^3_\infty /F^3E^3_\infty = E^{2,1}_\infty \subset E^{2,1}_2 = H^2\big(OutDer(\g),H^1((\Lambda^2\g)^\g,Z(\g))\big).$$
Since $a(d,X,Y)$ is identically zero the class of the Jacobiator in $H^2\big(OutDer(\g),H^1((\Lambda^2\g)^\g,Z(\g))\big)$ is trivial.
Thus the Jacobiator belongs to $F^2E^3_\infty = H^3(OutDer(\g),Z(\g))$ and we have the following statement.
\bpr
The Jacobiator of the crossed module of Lie algebras $$Der_{tw}(U(\g)) \stackrel{\partial}{\leftarrow} U(\g)$$ 
as a cohomology class in $$H^3\big(OutDer_{tw}(U(\g)),Z(\g)\big) = H^3\big(OutDer(\g)\ltimes (\Lambda^2\g)^\g,Z(\g)\big)$$ coincides with the image of the canonical class of the crossed module of Lie algebras
$$\xymatrix{Z(\g)\ar[r] & \g\ar[r]^(.34)\partial & Der(\g)\ar[r] & OutDer(\g)}$$
in $H^3(OutDer(\g),Z(\g))$.
\epr

\ere

\subsection{Non-separated twisted derivations}

Here we present a construction providing examples of non-separated twisted derivations. Let $H$ be an algebra.
Denote by $T(H[t])$ the tensor algebra of the vector space $H[t]$ of polynomials with coefficients in $H$. 
Consider the quotient algebra
$$E(H) = T(H[t])/\langle xyt^n - \sum_{i=0}^n C^n_i xt^i\ot yt^{n-i},\ 1t^n,\ \forall x,y\in H, n\geq 0\rangle.$$
Note that 
$$H\to E(H),\quad x\mapsto xt^0$$ is an embedding of algebras.
\newline
\ble
The assignment 
$$d:H[t]\to E(H),\quad d(xt^n) = xt^{n+1}$$ extends to an algebra derivation.
\ele
\bpf
All we need to check is that $d$ preserves the defining relations of $E(H)$.  This is a straightforward consequence of the property of binomial coefficients $C^{n+1}_i = C^n_{i-1}+C^n_i$. Indeed
$$d(\sum_{i=0}^n C^n_i xt^i\ot yt^{n-i}) = \sum_{i=0}^n C^n_i (d(xt^i)\ot yt^{n-i}+xt^i\ot d(yt^{n-i})) = \sum_{i=0}^n C^n_i (xt^{i+1}\ot yt^{n-i}+xt^i\ot yt^{n-i+1})$$ coincides with
$$d(xyt^n) = xyt^{n+1} = \sum_{i=0}^{n+1} C^{n+1}_i xt^i\ot yt^{n-i+1}.$$
\epf
\bre
The assignment $H\mapsto E(H)$ is a functor from the category of algebras to the category of algebras with a derivation (differential algebras), which is a left adjoint to the forgetful functor. In other words $E(H)$ is the {\em free differential algebra} on an algebra $H$.
\ere
From now on we will suppress the tensor sign and will use $d^i(x)$ instead of $xt^i$ when working with elements of $E(H)$.

Let now $H$ be a bialgebra and let $\phi\in H^{\ot 2}$ be a co-Hochschild 2-cocycle:
$$1\otimes \phi + (I\otimes\Delta)(\phi) = \phi\otimes 1 + (\Delta\otimes I)(\phi).$$
\ble
The inductively defined assignment
\be\lb{cop}\Delta(d^n(x)) = (d\ot I+I\ot d)(\Delta(d^{n-1}(x))) - [\phi,\Delta(d^{n-1}(x))],\quad x\in H\ee extends to a homomorphism of algebras $\Delta:E(H)\to E(H)\ot E(H)$, making $E(H)$ a bialgebra with the counit defined by
$$\ve(d^n(x)) = 0,\qquad x\in H,\ n>0.$$
\ele
\bpf
We will use induction to prove that $\Delta$ preserves the defining relations of $E(H)$: 
$$\Delta\big(d^n(xy)-\sum_{i=0}^n C^n_i d^i(x)d^{n-i}(y)\big) = \Delta(d^n(xy)) - \sum_{i=0}^n C^n_i \Delta(d^i(x))\Delta(d^{n-i}(y)) = $$
$$(d\ot I+I\ot d)(\Delta(d^{n-1}(xy))) - [\phi,\Delta(d^{n-1}(xy))] - $$
$$ - \sum_{i=0}^n C^n_i\big((d\ot I+I\ot d)(\Delta(d^{i-1}(x))) - [\phi,\Delta(d^{i-1}(x))]\big)\Delta(d^{n-i}(y))$$ $$- \sum_{i=0}^n C^n_i \Delta(d^i(x))\big((d\ot I+I\ot d)(\Delta(d^{n-i-1}(y))) - [\phi,\Delta(d^{n-i-1}(y))]\big) = $$
$$(d\ot I+I\ot d)\big(\Delta\big(d^{n-1}(xy) - \sum_{i=0}^{n-1} C^n_i d^i(x)d^{n-i-1}(y)\big)\big) - $$ $$- \big[\phi,\Delta\big(d^{n-1}(xy) - \sum_{i=0}^{n-1} C^n_i d^i(x)d^{n-i-1}(y)\big)\big] = 0.$$
The coasociativity of $\Delta$ can also be proved by induction:
$$(\Delta\ot I-I\ot\Delta)\Delta(d^n(x)) = (\Delta\ot I-I\ot\Delta)\big((d\ot I+I\ot d)(\Delta(d^{n-1}(x)))\big) - (\Delta\ot I-I\ot\Delta)([\phi,\Delta(d^{n-1}(x))])
=$$ $$(\Delta d\ot I+\Delta\ot d - d\ot\Delta-I\ot\Delta d)\Delta(d^{n-1}(x)) -$$ $$- [(\Delta\ot I)(\phi),(\Delta\ot I)(\Delta(d^{n-1}(x)))] + [(I\ot\Delta)(\phi),(I\ot\Delta)(\Delta(d^{n-1}(x)))] =$$
$$(d\ot I\ot I+I\ot d\ot I+I\ot I\ot d)(\Delta\ot I-I\ot\Delta)\Delta(d^{n-1}(x)) - $$ $$- [\phi\ot 1+(\Delta\ot I)(\phi),(\Delta\ot I)(\Delta(d^{n-1}(x)))] + [1\ot\phi+(I\ot\Delta)(\phi),(I\ot\Delta)(\Delta(d^{n-1}(x)))] = 0.$$
\epf
Note that the embedding $H\to E(H)$ is a homomorphism of bialgebras. 

\bth
The pair $(d,\phi)$ is a twisted derivation of the bialgebra $E(H)$. 
\eth
\bpf
Clearly, the conditions (\ref{cocd},\ref{normd}) are obvious. The condition (\ref{conjd}) follows from the definition of the coproduct (\ref{cop}).
\epf

\bex
Let $\ga = \langle x,y\rangle$ be an abelian 2-dimensional Lie algebra. Let $H = U(\ga) = k[x,y]$ be its universal enveloping algebra. The algebra $E(H)$ is the quotient of the free associative algebra $k\langle x,y,d(x),d(y),d^2(x),d^2(y),...\rangle$ by the ideal generated by 
$\sum_{i=0}^n C^n_i [d^i(x),d^{n-i}(y)]$ for all $n$.
Let $\phi = x\ot y\in H^{\ot 2}$ be an infinitesimal twist. The corresponding coproduct on $E(H)$ is defined by
$$\begin{array}{l} \Delta(x) = x\ot 1+1\ot x, \qquad\qquad\qquad \Delta(y) = y\ot 1+1\ot y,\\
\Delta(d(x)) = d(x)\ot 1+1\ot d(x), \qquad \Delta(d(y)) = d(y)\ot 1+1\ot d(y),\\
\Delta(d^2(x)) = d^2(x)\ot 1+1\ot d^2(x)-[x\ot y,d(x)\ot 1+1\ot d(x)], \\ \Delta(d^2(y)) = d^2(y)\ot 1+1\ot d^2(y)-[x\ot y,d(y)\ot 1+1\ot d(y)],\\
\end{array}$$
\eex

\bre
The above constructions is the infinitesimal analogue of the following. 
Let $H$ be an algebra.
Denote by $T(H[t^{\pm 1}])$ the tensor algebra of the vector space $H[t^{\pm 1}]$ of Laurent polynomials with coefficients in $H$. 
Consider the quotient algebra
$$A(H) = T(H[t])/\langle xyt^n - xt^n\ot yt^n,\ 1t^n-1,\ \forall x,y\in H, n\in\Z\rangle.$$
Note that 
$$H\to A(H),\quad x\mapsto xt^0$$ is an embedding of algebras.
\newline
\ble
The assignment 
$$f:H[t]\to A(H),\quad f(xt^n) = xt^{n+1}$$ extends to an algebra automorphism.
\ele
\bpf
All we need to check is that $f$ preserves the defining relations of $A(H)$.  This is straightforward $$f(xyt^n - xt^n\ot yt^n) = xyt^{n+1} - xt^{n+1}\ot yt^{n+1} = 0,\quad f(t^n-1) = t^{n+1}-1 = 0.$$ 
\epf
The assignment $H\mapsto A(H)$ is a functor from the category of algebras to the category of algebras with an automorphism, which is a left adjoint to the forgetful functor. In other words $A(H)$ is the {\em free algebra with an automorphism} on an algebra $H$.

From now on we will suppress the tensor sign and will use $f^i(x)$ instead of $xt^i$ when working with elements of $A(H)$.

Let now $H$ be a bialgebra and let $F\in H^{\ot 2}$ be a twist:
$$(1\otimes F)(I\otimes\Delta)(F) = (F\otimes 1)(\Delta\otimes I)(F).$$
\ble
The inductively defined assignment
\be\lb{copa}\Delta(f^n(x)) = F^{-1}((f\ot f)(\Delta(f^{n-1}(x))))F,\quad x\in H\ee extends to a homomorphism of algebras $\Delta:A(H)\to A(H)\ot A(H)$, making $A(H)$ a bialgebra with the counit defined by
$$\ve(f^n(x)) = \ve(x),\quad x\in H, n\in\Z.$$
\ele
\bpf
We will use induction to prove that $\Delta$ preserves the defining relations of $A(H)$. 
$$\Delta(f^n(xy)-f^n(x)f^n(y)) = \Delta(f^n(xy))-\Delta(f^n(x))\Delta(f^n(y)))  = $$
$$F^{-1}((f\ot f)(\Delta(f^{n-1}(xy))))F - F^{-1}((f\ot f)(\Delta(f^{n-1}(x))))FF^{-1}((f\ot f)(\Delta(f^{n-1}(y))))F = $$
$$F^{-1}((f\ot f)(\Delta(f^{n-1}(xy)-f^{n-1}(x)f^{n-1}(y))))F = 0.$$

The coasociativity of $\Delta$ can also be proved by induction:
$$(\Delta\ot I-I\ot\Delta)\Delta(f^n(x)) = (\Delta\ot I-I\ot\Delta)(F^{-1}((f\ot f)(\Delta(f^{n-1}(x))))F) =$$ 
$$(\Delta\ot I)(F)^{-1}(\Delta f\ot f)(\Delta(f^{n-1}(x)))(\Delta\ot I)(F) - (I\ot\Delta)(F)^{-1}(f\ot\Delta f)(\Delta(f^{n-1}(x)))(I\ot\Delta)(F) = $$
$$(\Delta\ot I)(F)^{-1}(F\ot 1)^{-1}(f\ot f)((\Delta\ot I)\Delta(f^{n-1}(x)))(F\ot 1)(\Delta\ot I)(F) - $$ $$- (I\ot\Delta)(F)^{-1}(1\ot\Delta)^{-1}(f\ot f)((I\ot\Delta)\Delta(f^{n-1}(x)))(1\ot\Delta)(I\ot\Delta)(F) = $$
$$(\Delta\ot I)(F)^{-1}(F\ot 1)^{-1}(f\ot f)((\Delta\ot I)\Delta(f^{n-1}(x)) - (I\ot\Delta)\Delta(f^{n-1}(x)))(F\ot 1)(\Delta\ot I)(F) = 0.$$
\epf
Note that the embedding $H\to A(H)$ is a homomorphism of bialgebras. 

\bth
The pair $(f,F)$ is a twisted automorphism of the bialgebra $A(H)$. 
\eth
\bpf
Clearly, the conditions (\ref{coc},\ref{norm}) are obvious. The condition (\ref{conj}) follows from the definition of the coproduct (\ref{copa}).
\epf

\ere

\subsection{Sweedler's Hopf algebra}

Here we describe twisted derivations of a non-commutative and non-cocommutative 4-dimensional Hopf algebra defined by Sweedler.
Recall that the Sweedler's Hopf algebra $H_4$ is the algebra
$$H_4 = k\langle g,x|\ g^2=1,\ x^2=0,\ gx+xg=0\rangle.$$
It is a Hopf algebra with the comultiplication
$$\Delta(g) = g\ot g,\quad \Delta(x) = 1\ot x + x\ot g,$$
the counit
$$\ve(g) = 1,\quad \ve(x) = 0,$$
and the antipode
$$S(g) = g^{-1},\quad S(x) = -x.$$

\ble\label{csw}
The second cohomology $H^2(H_4)$ of the Sweedler's Hopf algebra is one-dimensional and is generated by the element $\psi = x\ot gx\in H_4^{\ot 2}$.
\ele
\bpf
Lemma \ref{gpr} implies that $\psi$ is a 2-cocycle.

\epf

Let now $(d,\phi)$ be a twisted derivation of $H_4$. According to lemma (\ref{csw}) $\phi$ must be a multiple of $\psi = x\ot gx$.
Note that the element $\psi$ is $\Delta(H_4)$-invariant:
$$[\psi,\Delta(g)] = [x\ot gx,g\ot g] = xg\ot gxg - gx\ot x = 0,$$
$$[\psi,\Delta(x)] = [x\ot gx,1\ot x + x\ot g] = 0.$$
Thus up to gauge equivalences twisted derivations are separated:
$$OutDer_{tw}(H_4) = OutDer_{bialg}(H_4)\ltimes OutDer^0_{tw}(H_4).$$
Now all what remains to analise to have a full control of twisted derivations is the Lie algebra of outer bialgebra derivations. Since the space $Prim(H_4)$ of primitive elements is zero, $H_4$ does not have inner bialgebra derivations:
$$OutDer_{bialg}(H_4) = Der_{bialg}(H_4).$$
For a bialgebra derivation $d:H_4\to H_4$ the element $g^{-1}d(g)$ has to be primitive. Indeed,
$$\Delta(d(g)) = d(g)\ot g + g\ot d(g)$$ and $$\Delta(gd(g)) = g^{-1}d(g)\ot 1 + 1\ot g^{-1}d(g).$$
Thus for a bialgebra derivation $d:H_4\to H_4$ we have $d(g)=0$. 
The value $d(x)$ is $g$-primitive:
$$\Delta(d(x)) = 1\ot d(x) + d(x)\ot g + x\ot d(g) = 1\ot d(x) + d(x)\ot g.$$
Hence any bialgebra derivation of $H_4$ is proportional to 
$$d(g) = 0,\quad d(x) = x.$$

\bth
The algebra $OutDer_{tw}(H_4)$ of twisted derivations of the Sweedler's Hopf algebra is a two-dimensional non-abelian Lie algebra.
\eth
\bpf
We have shown that $OutDer_{tw}(H_4)$ is a 2-dimensional vector space with generators $(d,0), (0,\psi)$. The non-triviality of the bracket can be seen directly
$$[(d,0),(0,\psi)] = (0,(d\ot I+I\ot d)(\psi))$$
with $$(d\ot I+I\ot d)(\psi) = (d\ot I+I\ot d)(x\ot gx) = 2x\ot gx = 2\psi.$$
\epf

\bre 
Integrating the bialgebra derivation $d:H_4\to H_4$ we get a one parameter family of bialgebra automorphisms:
$$f_{c}:H_4\to H_4,\quad f_{c}(g) = g,\quad f_c(x) = cx,\quad c\in k^*$$
$f_c\circ f_{c'} = f_{cc'}$

Similarly exponentiating the 2-cocycle $\psi$ gives a one parameter family of invariant twists
$$\Phi_a = 1\ot 1 + ax\ot gx,\quad a\in k.$$

Note that 
$$f_c\circ f_{c'} = f_{cc'},\quad \Phi_a\Phi_{a'} = \Phi_{a+a'},\quad (f_c\ot f_c)(\Phi_a) = \Phi_{c^2a}$$
So that the group $Out_{tw}(H_4)$ of twisted automorphisms contains a subgroup isomorphic to $k^*\ltimes k$ (with the $k^*$-action on $k$ given by $c*a = c^2a$).
\ere

\section{Appendix. Co-Hochschild cohomology of bialgebras.}\lb{coho}

Following Drinfeld \cite{dr} consider the
complex $C^*(H) = (H^{\otimes *},\partial)$ with the differential $\partial:H^{\otimes n}\to H^{\otimes n+1}$ defined by
\begin{equation}\label{tangcoh}
\partial(X) = 1\otimes X +\sum_{i=1}^n(-1)^i(I^{\otimes i-1}\otimes\Delta\otimes I^{\otimes n-i-1})(X) + (-1)^{n+1}(X\otimes 1).
\end{equation}
We call this complex the {\em co-Hochschild complex} of $H$ and its cohomology {\em co-Hochschild cohomology} of $H$. The motivation for the choice of name is the following. 
There is a canonical map $C^*(H)\to CH^*(H^\vee,k)$ into the standard Hochschild complex of the dual algebra $H^\vee$ with coefficients in the trivial module $k$. This map is an isomorphism if $H$ is finite-dimensional.
\bex
Let $H = k[G]$ be the group algebra of a finite group $G$. The dual $k[G]^\vee = k(G)$ is a semi-simple algebra. Then the isomorphism of complexes $C^*(k[G])\to CH^*(k(G),k)$ implies that the co-Hochschild cohomology $H^*(k[G])$ is trivial ($k$ in degree 0 and zero otherwise). The last is true without assuming that $G$ is finite (and can bee proved by establishing a contracting homotopy for the complex $C^*(k[G])$). 
\eex 

Note that we only use coalgebra structure in the definition of the co-Hochschild complex. Thus it is defined for a coalgebra. 

The co-Hochschild complex $C^*(H)$ is a differential-graded algebra with respect to the {\em cup product}:
$$\cup:C^m(H)\otimes C^n(H)\to C^{m+n}(H),\quad X\cup Y = X\ot Y.$$
Indeed, the property $\partial(X\ot Y) = \partial(X)\ot Y + (-1)^{m}X\ot\partial(Y)$ can be checked directly. 
The cup product induces a product on co-Hochschild cohomology.
\bpr
The cup product on co-Hochschild cohomology $H^*(H)$ is graded commutative:
$$[X]\cup[Y] = (-1)^{mn}[Y]\cup[X],\quad [X]\in H^m(H),\ Y\in H^n(H).$$
\epr
\bpf
The construction is similar to the one for Hochschild complex from \cite{g0}. 
For $X\in H^{\ot m}$, $Y\in H^{\ot n}$ and $i=1,...,m$ define
$$X\circ_iY = (I^{\ot i-1}\ot\Delta^{(n-1)}\ot I^{\ot m-i})(X)(1_{H^{\ot i-1}}\ot Y\ot  1_{H^{\ot m-i}}),$$
where $\Delta^{(n-1)}:H\to H^{\ot n}$ is the iterated coproduct. The operations $\circ_i$ satisfy to the axioms of a pre-Lie system of \cite{g0}:
$$(X\circ_iY)\circ_jZ = \left\{\begin{array}{ll} (X\circ_jZ)\circ_{i+p} Y,& j<i\\ X\circ_i(Y\circ_jZ),& i\leq j\end{array}\right.$$
for $X\in H^{\ot m}$, $Y\in H^{\ot n}$ and $Z\in H^{\ot p}$.
Thus, according to \cite{g0}, the operation
$$X\circ Y = \sum_{i=1}^m(-1)^{ni}X\circ_iY$$ is a homotopy for commutativity of the cup product:
$$Y\cup X - (-1)^{nm}X\cup Y = \partial(X)\circ Y + (-1)^{n-1}X\circ\partial(Y) - (-1)^{n-1}\partial(X\circ Y).$$
\epf

\bre\lb{ger}
The bracket 
$$[[X,Y]] = X\circ Y - (-1)^{mn}Y\circ X$$ endows $H^*(H)$ with a graded Lie bracket
$$[[\ ,\ ]]:H^m(H)\ot H^n(H)\to H^{m+n-1}(H),$$ 
which together with the cup product turn $H^*(H)$ into a Gerstenhaber algebra.
\ere

The following gives a description of co-Hochschild cohomology of a universal enveloping algebra in characteristic zero.
\begin{prop}\label{chu}
For a universal enveloping algebra $H=U(\g)$ the alternation map $Alt_n:H^{\otimes n}\to \Lambda^n H$ induces an
isomorphism of the n-th co-Hochschild cohomology  and $\Lambda^n\g$.
\end{prop}
\begin{proof}
(Sketch of, for details see \cite{dr}):
\newline
By Poincare-Birkhoff-Witt theorem, the map
$$S^*(\g)\to U(\g),\quad\quad x_1\ot...\ot x_n\mapsto Sym_n(x_1...x_n),\quad x_i\in\g$$
 is an isomorphism of coalgebras. 
 The complex (\ref{tangcoh}) for $H = S^*(\g)$ breaks into (tensor products of) pieces:
$$S^n(\g)\to \bigoplus_{i_1+i_2=n}S^{i_1}(\g)\otimes S^{i_2}(\g)\to ...\to
\bigoplus_{i_1+...+i_s=n}\otimes_{j=1}^sS^{i_j}(\g)\to ...\to \g^{\otimes n}
$$
The  $n$-th piece is isomorphic to the cochain complex of the simplicial $n$-cube tensored (over the symmetric group $S_n$)
with $\g^{\otimes n}$.
\end{proof}

\bre\lb{chho}
The map $Alt_*:H^{\otimes *}\to H^{\otimes *}$ from proposition \ref{chu} is  a map of co-Hochschild complex(es). 
Moreover $Alt_*$ is homotopic to the identity, i.e. there is a collection of maps $a_n: H^{\otimes n}\to H^{\otimes n-1}$ such that
$$I-Alt_n = \partial\circ a_n + a_{n+1}\circ\partial.$$
\ere

\bre
Under the isomorphism $H^*(U(\g))\to\Lambda^*(\g)$ the Gerstenhaber bracket on $H^*(U(\g))$ corresponds to the {\em Schouten} bracket on $\Lambda^*(\g)$:
$$[[x_1\wedge...\wedge x_m,y_1\wedge...\wedge y_n]] = \sum_{i,j}(-1)^{i+j}[x_i,y_j]\wedge x_1\wedge...\wedge \hat x_i\wedge...\wedge x_m\wedge y_1\wedge...\wedge \hat y_j\wedge...\wedge y_n.$$ Here $\hat x$ means that $x$ is omitted in the exterior product.
\ere

\bre
Note that the first cohomology $H^1(H)$ of a bialgebra coincides with the space $Prim(H)$ of its primitive elements, which forms a Lie algebra with respect to the commutator. The embedding $Prim(H)\subset H$ induces a homomorphism of bialgebras $U(Prim(H))\to H$, which in its turn gives rise to a map
$$\Lambda^*(Prim(H))\to H^*(H).$$
\ere

The following describes the cohomology of the the subcomplex of $H$-invariant elements of (\ref{tangcoh})
in the case of universal enveloping algebra $H$.
\begin{lem}\lb{invco}
For a universal enveloping algebra $H=U(\g)$ the alternation map $Alt_n:(H^{\otimes n})^H\to (\Lambda^n H)^H$ induces
an isomorphism of the n-th cohomology of the subcomplex of $H$-invariant elements of (\ref{tangcoh}) and the space of
$\g$-invariant skew-symmetric tensors $(\Lambda^n\g)^\g$.
\end{lem}
\begin{proof}
The coalgebra isomorphism between $U(\g)$ and $S^*(\g)$ is $\g$-invariant. The isomorphism between the degree $n$ component of 
(\ref{tangcoh}) and the cochain complex of the simplicial $n$-cube tensored with $\g^{\otimes n}$ is natural in $\g$
and in particular $\g$-invariant.
\end{proof}

\bre
Note that the Schouten bracket on the $(\Lambda^n\g)^\g$ is zero. 
\ere

Now let $K\subset H$ be an embedding of Hopf algebras. 
Clearly $C^*(K)\subset C^*(H)$ is a subcomplex.
More generally define
$$C^i(H)_m = \left\{\begin{array}{ll}K^{\ot i}, & i\leq m\\ K^{\ot m}\ot H^{\ot i-m},& i>m\quad . \end{array}\right. $$
Then 
$$C^*(H) = C^*(H)_0\supset C^*(H)_1\supset ... $$ is a decreasing filtration by subcomplexes
with 
$\cap_{m=0}^\infty C^*(H)_m = C^*(K)$.

Let $g$ be a group-like element of a Hopf algebra $H$. An element $x\in H$ is called {\em $g$-primitive} if
$$\Delta(x) = 1\ot x + x\ot g.$$ Denote by $Prim_g(H)$ the space of $g$-primitive elements of $H$.
\ble\lb{gpr}
Let $g_i, i=1,...,m$ be a collection of group-like elements in $H$ such that $g_1...g_m=1$. Then for any $x_i\in Prim_{g_i}(H)$ 
$$F = x_1\ot g_1x_2\ot g_1g_2x_3\ot...\ot g_1...g_{m-1}x_m\in H^{\ot m}$$ is a co-Hochschild cocycle.
\ele
\bpf
Lets first drop the assumption $g_1...g_m=1$. We can prove by induction that $\partial(F) = F\ot(g_1...g_m-1)$. Indeed, 
$$\partial(F) = 1\ot x_1\ot g_1x_2\ot g_1g_2x_3\ot...\ot g_1...g_{m-1}x_m + $$
$$\sum_{i=1}^n(-1)^i x_1\ot g_1x_2\ot...\ot(g_1...g_{i-1}\ot g_1...g_{i-1}x_i + g_1...g_{i-1}x_i\ot g_1\ot g_1...g_i)\ot...\ot g_1...g_{m-1}x_m + $$ $$+ (-1)^{n+1}( x_1\ot g_1x_2\ot...\ot g_1...g_{m-1}x_m\otimes 1) = F\ot(g_1...g_m-1).$$ Now the lemma follows.
\epf

\end{document}